\documentclass[12pt]{amsart}

\usepackage{enumerate, amsmath, amsthm, amsfonts, amssymb, xy,  mathrsfs, graphicx, paralist, fancyvrb, ytableau, tikz-cd,epic,wasysym,upgreek}
\usepackage[usenames, dvipsnames]{xcolor}
\usepackage[margin=1in]{geometry} 
\usepackage[bookmarks, colorlinks=true, linkcolor=blue, citecolor=blue, urlcolor=blue]{hyperref}

\input xy
\xyoption{all}

\numberwithin{equation}{subsection}
\newtheorem{theorem}[equation]{Theorem}

\newtheorem{proposition}[equation]{Proposition}
\newtheorem{lemma}[equation]{Lemma}
\newtheorem{corollary}[equation]{Corollary}

\theoremstyle{definition}
\newtheorem{rmk}[equation]{Remark}
\newenvironment{remark}[1][]{\begin{rmk}[#1] \pushQED{\qed}}{\popQED \end{rmk}}
\newtheorem{eg}[equation]{Example}
\newenvironment{example}[1][]{\begin{eg}[#1] \pushQED{\qed}}{\popQED \end{eg}}
\newtheorem{defn}[equation]{Definition}

\makeatletter
\@addtoreset{equation}{section}
\renewcommand{\thesubsection}{%
  \ifnum\c@subsection<1 \@arabic\c@section
  \else \thesection.\@arabic\c@subsection
  \fi
}
\makeatother

\newcommand{\fm}{\mathfrak{m}}

\renewcommand{\phi}{\varphi}

\makeatletter
\def\Ddots{\mathinner{\mkern1mu\raise\p@
\vbox{\kern7\p@\hbox{.}}\mkern2mu
\raise4\p@\hbox{.}\mkern2mu\raise7\p@\hbox{.}\mkern1mu}}
\makeatother

\newtheorem*{mainresult}{Main Result}
\newenvironment{Main Result}[1][]{\begin{mainresult}[#1] \pushQED{\qed}}{\popQED \end{mainresult}}

\title{Herzog Class of Codimension $3$ Licci Monomial Ideals}
\author{Suhas Vadan Gondi}
\address{Department of Mathematics, University of California San Diego}
\email{sgondi@ucsd.edu}

\begin{document}
\begin{abstract}
    Let $S$ be the polynomial ring in three variables with homogeneous maximal ideal $\fm$. Licci ideals are ideals that can be linked by a finite number of steps to a complete intersection. All licci ideals admit a rigid deformation, and two licci ideals are said to be in the same Herzog class if they have a common rigid deformation. Guerrieri-Ni-Weyman gave a parametrization for the Herzog classes of codimension $3$ licci ideals in terms of pairs of partitions associated to Schur functors. In this paper, we give an explicit formula that determines the Herzog class of a given $\fm$-primary licci monomial ideal.     
\end{abstract}
\maketitle
\section{Introduction}

Linkage theory provides a means of studying ideals through other ideals to which they are linked. Its foundations were established by Peskine and
Szpiro \cite{PeskineSzpiro1974}, and its algebraic structure was developed further
by Huneke and Ulrich \cite{HunekeUlrich1987}. An ideal is \emph{licci} if it can be
connected to a complete intersection by a finite sequence of links.

Roughly, a deformation of an ideal is a ring--ideal pair from which the original
pair is recovered by quotienting by a sequence that is regular in both the
ambient ring and its quotient. An ideal is said to be rigid if it defines a rigid algebra. Buchweitz proved the strong
unobstructedness of licci ideals \cite{Buchweitz81}, and Herzog's results on
strongly unobstructed algebras imply that they admit rigid deformations, 
unique up to isomorphism and trivial deformation \cite{Herzog80}.
Two licci ideals belong to the same \emph{Herzog class} if they have a
common rigid deformation. Describing these classes amounts to finding
generic models for families of licci ideals.

In codimension two, every perfect ideal is licci, and the Hilbert--Burch
theorem provides such models in terms of ideals of maximal minors of generic
$n\times(n-1)$ matrices. In particular, there is a one-to-one correspondence between Herzog classes and the number of minimal generators $n\geq2$. In codimension three,
the number of generators and the Cohen--Macaulay type no longer determine
the Herzog class, and the problem of classifying and determining Herzog classes is significantly harder. Let $S=\mathbb{C}[x,y,z]$ and $\fm=(x,y,z).$ An $\fm$-primary monomial ideal is an ideal that can be generated by monomials and whose radical is $\fm$. The main result of this paper is the following. See Theorem \ref{thm:closed-new} for the specific statement.

\begin{mainresult}
    We provide a closed formula that determines the Herzog class of an $\fm$-primary licci monomial ideal $I$ of $S$ in terms of its minimal generators. 
\end{mainresult}

We now briefly review the development of the techniques that we were fortunate enough to have at our disposal in solving this problem. The starting point is the groundbreaking work of Weyman, who began studying generic free resolutions of length three in 1989 \cite{Weyman1989}, generalizing the Buchsbaum--Eisenbud structure theorem for grade three Gorenstein ideals in terms of Pfaffians of alternating matrices \cite{BuchsbaumEisenbud1977}.  For a fixed resolution format, he began the construction of the generic ring
$\widehat R_{\mathrm{gen}},$ which carries a complex from which resolutions of
that format are obtained by specialization. His work was completed in 2018 \cite{Weyman2018}, revealing a beautiful connection between the structure theory of free resolutions of the form $$ 0 \to \mathbb{F}_3\xrightarrow{d_3} \mathbb{F}_2\xrightarrow{d_2} \mathbb{F}_1\xrightarrow{d_1} R\to R/I $$ 
and representation theory, by associating a Kac--Moody Lie algebra whose diagram is $T$-shaped and depends on the format of the resolution. The representations of this Lie algebra organize the
additional structure carried by a resolution, extending its
differentials and multiplicative structure. Sam and Weyman~\cite{SamWeyman2021} gave geometric constructions of resolutions
in Dynkin formats using Schubert varieties. 

These tools were recently applied to linkage by Guerrieri, Ni, and
Weyman \cite{GuerrieriNiWeyman2026Linkage}, building on their earlier work \cite{GuerrieriNiWeyman2026}. Let $\mathbb F$ be a minimal free resolution of $R/I,$ where $I$ is a grade-three perfect ideal in a local ring $(R,\mathfrak n)$ containing a field of
characteristic zero. Specializing the generic ring gives higher
structure maps associated to graded representations of the corresponding
Kac--Moody Lie algebra. We refer the reader to \cite{GuerrieriNiWeyman2026, FilippiniGuerrieri2025, GuerrieriWeyman2023} for a detailed account of these higher structure maps. The ideal $I$ is licci if and only if the first higher structure map $w^{(1)}$ is nonzero
modulo $\mathfrak n$. Moreover, its Herzog class is determined by the lowest
extremal component, in Bruhat order, on which the map contains a unit \cite{GuerrieriNiWeyman2026Linkage, LinkageSchutFuctors}. Writing this component in terms of Schur
functors gives a pair of partitions $(\lambda,\mu)$, called the
\emph{decoration} of $I$. In the notation of~\cite{LinkageSchutFuctors}, it has the
form $S_\lambda F_1\otimes S_\mu F_3^*$. 

Guerrieri, Ni, and
Weyman describe the resulting parametrization of Herzog
classes and give explicit rules for how the partitions change under
linkage. In particular, their work gives an algorithmic way of computing the Herzog class of a given codimension $3$ licci ideal provided that we know precisely how it is linked to a complete intersection. However, there is no direct formula to obtain the Herzog class of a given ideal. In this article, we address this problem for a distinguished family of licci ideals, namely $\fm$-primary monomial licci ideals. 

Huneke and Ulrich \cite{MonomialLicci} describe an explicit combinatorial structure on the unique monomial generating set of an $\fm$-primary licci monomial ideal $I$. Using this, we construct an explicit linkage chain that connects $I$ to a complete intersection, and use the results from \cite{LinkageSchutFuctors} to obtain a recursive formula for its Herzog class. This recursion is dynamic, meaning that the coefficients vary from step to step according to the linkage data. Perhaps the most novel part of this paper is to solve this recurrence relation to obtain a closed formula in terms of combinatorial data attached to $I.$ We hope that the techniques and results developed here will also prove useful in obtaining direct formulas for the Herzog classes of other licci ideals.

\subsection{Organisation of the Paper}
In Section \ref{Section:Prelims}, we introduce the relevant notation and prerequisites. In particular, we recall the definitions of licci ideals, deformations, and Herzog classes, and summarize the results of \cite{MonomialLicci} on $\fm$-primary licci monomial ideals and of \cite{LinkageSchutFuctors} on the combinatorics of decorations. In Section \ref{Section:CombinatorialData}, we introduce the combinatorial constructions associated to these ideals that are used in the following section. Section \ref{Section:Mainresult} contains the statement and proof of the main result. We prove the recursive formula in Subsection \ref{Subsec:rec} and use it to prove the closed formula in Subsection \ref{Subsec:Closed}.   

\subsection{Acknowledgments}
I am grateful to Lorenzo Guerrieri and Jerzy Weyman for introducing me to this area, and for several useful discussions. I would also like to thank Jacinta Torres for helpful conversations. Finally, I thank Steven Sam for funding my visit to Jagiellonian University, where this project began.

\section{Preliminaries and Notation}\label{Section:Prelims}
For this paper, we work over the polynomial ring $S=\mathbb{C}[x,y,z]$ with maximal ideal $\fm=(x,y,z).$ For convenience, we state the preliminary results specifically for this setup, although the results hold in greater generality.

\subsection{Licci Ideals, Deformations and Herzog Classes}
Two proper ideals $I,J\subset S$ of height three are \emph{directly linked}
if there is a regular sequence $f_1,f_2,f_3\in I\cap J$ such that

$$
 I=(f_1,f_2,f_3):J \qquad\text{and}\qquad J=(f_1,f_2,f_3):I.
$$ Here, the colon operator is defined by $$(f_1,f_2,f_3): I=\{w\in S \mid wI\subset (f_1,f_2,f_3)\}.$$ We denote a direct link by $I\sim J$. An ideal is \emph{licci} if a finite sequence of direct links connects it
to a complete intersection. 

Following \cite[\S 2.4]{LinkageSchutFuctors},
a ring-ideal pair $(R,J)$ is a \emph{deformation} of $(S,I)$ if
$$
 (S,I)=\bigl(\frac{R}{(\boldsymbol u)},
                  \frac{J+(\boldsymbol u))}{(\boldsymbol u)}\bigr),
$$
where $\boldsymbol u$ is regular in both $R$ and $R/J$.
A deformation is \emph{trivial} if, up to isomorphism, it is obtained by
adjoining formal power series variables:
$$
 \bigl(S[[t_1,\ldots,t_r]],\,IS[[t_1,\ldots,t_r]]\bigr).
$$
We say that $I\subset S$ is rigid if it only admits trivial deformations. By results of Buchweitz and Herzog \cite{Buchweitz81,Herzog80}, every licci ideal admits a rigid
deformation, unique up to isomorphism and trivial deformation. Two licci ideals belong to the same \emph{Herzog class} if they admit a common rigid deformation.

\subsection{$\fm$-primary Licci Monomial Ideals}\label{mprimarySection}

An ideal $I\subset S=\mathbb C[x,y,z]$ is a \emph{monomial ideal} if it
can be generated by monomials, and it is \emph{$\mathfrak m$-primary}
if $\sqrt I=\mathfrak m=(x,y,z)$. Every $\mathfrak m$-primary monomial
ideal has a unique \emph{standard form}
$$
 I=(x^\alpha,y^\beta,z^\gamma)+I^{\#},
$$
where $\alpha,\beta,\gamma$ are positive integers and the minimal monomial generators
of $I^{\#}$, together with $x^\alpha,y^\beta,z^\gamma$, form the minimal monomial
generating set of $I$. In particular, $I^{\#}=0$ if and only if $I$ is
a complete intersection. If $I^{\#}\ne0$, let
$x^a y^b z^c$ be the greatest common divisor of its minimal
monomial generators. Then
$$
 I^{\#}=x^a y^b z^c K,
$$
where $K$ is uniquely determined and its minimal monomial generators have
greatest common divisor $1$. Here
$0\le a<\alpha$, $0\le b<\beta$, and $0\le c<\gamma$.
If $I^{\#}$ is principal, then $K=S$, otherwise $K$ is a proper monomial
ideal of height at least two.

Set $I^{\{0\}}=I$ and define
$$
 I^{\{1\}}=
 \begin{cases}
 S,& I^{\#}=0,\\[2pt]
 (x^{\alpha-a},y^{\beta-b},z^{\gamma-c})+K,& I^{\#}\ne0.
 \end{cases}
$$

Then, by \cite[Lemma 2.5]{MonomialLicci} $I$ and $I^{\{1\}}$ are linked by the following double link: 
$$ I \sim  (x^\alpha, y^\beta, z^\gamma):I \sim  I^{\{1\}},
$$ where the first link is via the regular sequence $\{x^\alpha, y^\beta, z^\gamma\}$ while the second link is via the regular sequence $\{x^{\alpha-a}, y^{\beta-b}, z^{\gamma-c}\}.$ If $I^{\#}$ is principal, the first link is a direct link between $I$ and a complete intersection. If $I^{\#}=0,$ then $I$ is already a complete intersection.

For $r\ge2$, define recursively
$$
 I^{\{r\}}=\bigl(I^{\{r-1\}}\bigr)^{\{1\}},
 \qquad\text{provided }I^{\{r-1\}}\ne S.
$$

We caution that $(x^{\alpha-a},y^{\beta-b},z^{\gamma-c})+K$ may not be the standard form of $I^{\{1\}},$ as it is possible that $K$ contains pure powers. Thus, we first need to express it in standard form before defining $I^{\{2\}}.$ 

\begin{theorem}(\cite[Theorem 2.6]{MonomialLicci})
\label{thm:hu-criterion}
Let $I\subset S=\mathbb C[x,y,z]$ be an $\mathfrak m$-primary monomial
ideal. The following conditions are equivalent:
\begin{enumerate}
 \item $I_{\mathfrak m}$ is licci in $S_{\mathfrak m}$.
 \item $I$ can be linked to a complete intersection using monomial
 regular sequences in $S$.
 \item $I^{\{k\}}=S$ for some $k\ge1$.
\end{enumerate}
\end{theorem}

\subsection{Decorations and Combinatorics of Links}
In this section, we summarize the required results from Sections 3.3 and 3.4 of \cite{LinkageSchutFuctors} for our specific situation. We only talk about the combinatorics involved without alluding to the higher structure maps and structure theory for length $3$ free resolutions, which play a central role in the proofs of these results. 

To each Herzog class of codimension $3$ licci ideals we can univocally assign a pair of partitions $(\lambda, \mu).$ In particular, this assigns a pair of partitions to each height $3$ licci ideal $I\subset S,$ and we refer to this as the \emph{decoration} of $I$. 

Let $I$ be a height $3$ licci ideal and $(\lambda, \mu)$ its decoration. Let $\mu(I)$ denote the minimal number of generators required to generate $I,$ and let $t(I)$ denote the type of $I$. 

The following properties hold for the decoration $(\lambda, \mu)$ of $I$:
\begin{enumerate}
    \item Let $b$ denote the minimal number of generators required to generate $I.$ Then $\lambda$ has exactly $b$ parts.
    \item Let $t$ denote the type of $I$. Then $\mu$ has exactly $t$ parts.
    \item Let $\kappa=\sum\mu_i.$ Then $\sum \lambda_i=2\kappa+1,$ and we will refer to $\kappa$ as the \emph{level} of I. We use $\kappa(I)$ sometimes to indicate which ideal's level we are talking about.
    \item $\sum_i\lambda_i^2+\sum_j\mu_j^2= (k+1)^2$ holds \cite[\S 5 ]{LinkageSchutFuctors}
    \item The decoration of a codimension $3$ complete intersection is $((1^3), (1)).$ 
\end{enumerate}

Now, assume that $I$ is an $\fm$-primary licci monomial ideal. $I$ has a unique minimal generating set consisting of monomials. Since the number of parts of $\lambda$ equals the number of minimal generators, we can try to create an assignment between the monomial generators and parts of $\lambda.$ There is a canonical way of doing this using special generating systems (SGS) \cite[Definition 3.3]{LinkageSchutFuctors}, that we now explain. Let $I=(x^a, y^b, z^c),$ which is a complete intersection with decoration $((1^3), (1)). $ In this case, an SGS is given by assigning $1$ to each monomial generator. Using this as a starting point, we may inductively define monomial SGSs for other licci ideals, using the combinatorial formula for decorations in linkage that we state below. 

Suppose that $I$ is an $\fm$-primary licci monomial ideal with decoration $(\lambda, \mu),$ and with an assignment between the parts of $\lambda$ and the minimal monomial generators of $I$ (an SGS). Let $x^{a_x}, y^{a_y}, z^{a_z}$ be the pure power generators and let $\lambda_x,$ $\lambda_y$ and $\lambda_z$ be their respective assigned parts. Consider the regular sequence $\{x^{b_x}, y^{b_y}, z^{b_z}\} \subset I$ where $b_c\geq a_c,$ for $c=x,y,z$ and let $$J=\{x^{b_x}, y^{b_y}, z^{b_z}\}: I.$$ 

For $c=x, y,z,$ define $$\lambda_c'=  \begin{cases}
 \lambda_c,& b_c=a_c,\\[2pt]
 0,& b_c>a_c.
 \end{cases}$$ Let $p=\kappa(I)+1-\lambda_x'-\lambda_y'-\lambda_z'.$ 
 
 \begin{proposition}\cite{LinkageSchutFuctors}\label{LinkageCombinatoricsTheorem} The decoration of $J$ is given by $(\lambda', \mu'),$ where  $$\lambda' = \mathrm{sort}(p+\lambda_x', p+\lambda_y', p+\lambda_z', \mu),$$ $$\mu'=\lambda \backslash \{\lambda_x', \lambda_y', \lambda_z'\}.$$ \end{proposition} By sort, we mean to rearrange in decreasing order so that we obtain a partition. We note that $\kappa(J)=\kappa(I)+p.$ Further, by Lemmas 3.13 and 3.15 of \cite{LinkageSchutFuctors}, there is an assignment between the parts of $\lambda'$ and minimal monomial generators of $J$ induced by the assignment for the generators of $I$ that assigns $p+\lambda_c'$ to the pure powers in $J$ corresponding to $c=x,y,z.$ 

 \begin{remark}
     There is a technical point regarding the application of \cite[Lemma 3.15]{LinkageSchutFuctors} to obtain an SGS for $J$ that we now address. The Lemma requires the monomials to be of specific degrees for it to be applicable, and this may or may not be achievable directly with respect to some appropriate grading on the variables. If it is not possible, we may add variables to our ring and write the pure powers as square-free products over disjoint variables, so that we have the freedom to assign different degrees to different variables in order to satisfy the degree constraints in the Lemma. We hence obtain an SGS in this bigger ring, and we can then specialize back to $S$ and use the fact that specialization preserves SGSs to obtain an SGS for $J$.   
 \end{remark}

\begin{example}
Let
\[
I=(x^2,y^2,z,xy)\subseteq \mathbb C[x,y,z].
\]
Its decoration is
\[
(\lambda,\mu)=((2,1,1,1),(1,1)),
\]
with the following assignment of parts to its monomial generators:
\[
\begin{array}{c|cccc}
\text{generator}&x^2&y^2&z&xy\\ \hline
\text{assigned part}&1&1&2&1.
\end{array}
\]
Indeed, this assignment is obtained by linking the complete
intersection $(x,y,z)$ via $(x^2,y^2,z)$.

Now link $I$ using the regular sequence $x^2,y^2,z^2$.
A direct calculation gives
\[
J=(x^2,y^2,z^2):I=(x^2,y^2,z^2,xz,yz).
\]
The elements $x^2$ and $y^2$ are minimal generators of $I$,
whereas $z^2\in\mathfrak m I$. Thus
\[
\lambda'_x=1,\qquad
\lambda'_y=1,\qquad
\lambda'_z=0.
\]
Since $\kappa(I)=|\mu|=2$, we obtain
\[
p=\kappa(I)+1-\lambda'_x-\lambda'_y-\lambda'_z
 =3-1-1=1.
\]
Proposition~2.3.1 therefore yields
\[
\begin{aligned}
\lambda'
&=\operatorname{sort}(p+\lambda'_x,
                     p+\lambda'_y,
                     p+\lambda'_z,\mu)
 =(2,2,1,1,1),\\
\mu'
&=\lambda\setminus\{1,1\}
 =(2,1).
\end{aligned}
\]
 Consequently, the decoration of $J$ is
\[
{((2,2,1,1,1),(2,1))}.
\]
The new pure-power generators $x^2,y^2,z^2$ receive the
parts $2,2,1$, respectively, and
$\kappa(J)=3=\kappa(I)+p$.
\end{example}

\section{Combinatorial Data Attached to an $\fm$-primary Licci Monomial Ideal}\label{Section:CombinatorialData}
In this section, we define some combinatorial statistics on $\fm$-primary licci monomial ideals that we will use to give a formula for computing their Herzog class later in the paper.

Let $I$ be an $\fm$-primary licci monomial ideal. Theorem \ref{thm:hu-criterion} tells us what the minimal monomial basis of $I$ looks like. If we start by isolating the pure powers followed by dividing the remaining generators by their gcd and repeat this process, we eventually exhaust all the generators. Recording these sets of pure powers using $V_i$ and the gcd using $x^{a_i}y^{b_i}z^{c_i}$ in each step gives us the following proposition.

\begin{proposition}\label{DataProp}
    Suppose $I$ is an $m$-primary licci monomial ideal. There exist a nonnegative integer $k$, triples $
 (a_i,b_i,c_i)\in\mathbb Z_{\ge0}^3\setminus\{(0,0,0)\}$ for $1\leq i\leq k,$
and finite sets of monomials $V_1,\ldots,V_{k+1}$ such that the minimal generating monomial basis for $I$ is given by
$$\bigcup_{i=1}^{k+1}\{U_i v\mid v\in V_i\},$$
where
 $U_i=\prod_{j=i}^{k}x^{a_j}y^{b_j}z^{c_j}$ and
the empty product gives $U_{k+1}=1$.
Each $V_i$ consists of positive pure powers of $x$, $y$, and $z$,
with at most one power of each variable, except that $V_1$ is also
allowed to be $\{1\}$. Moreover, $V_{k+1}$ contains the three pure power generators of $I.$ 

\end{proposition}

\begin{example}
\label{ex:monomial-layer-data}
Consider
$$
 I=(x^{12},y^8,z^8,x^9z,x^2z^7,x^8z^2,x^6yz^3,
                    x^2y^6z^3,x^2yz^6,x^5y^3z^4,x^3y^5z^4).$$
After removing the three pure-power generators, the remaining generators
have greatest common divisor $x^2z$. Dividing by this monomial gives
$$
 (x^7,z^6,x^6z,x^4yz^2,y^6z^2,yz^5,x^3y^3z^3,xy^5z^3).
$$
Repeating the procedure yields the nested expression
$$
 I=(x^{12},y^8,z^8)
   +x^2z\Bigl((x^7,z^6)
   +z\bigl((x^6)+yz((x^4,y^5,z^3)+xy^2z(x^2,y^2))\bigr)\Bigr).
$$
Thus $k=4$. Indexing from the innermost layer to the outermost, we obtain
$$
\begin{array}{|c|c|c|c|}
\hline
 i & V_i & (a_i,b_i,c_i) & U_i\\ \hline
 1 & \{x^2,y^2\}       & (1,2,1) & x^3y^3z^4\\ 
 2 & \{x^4,y^5,z^3\}   & (0,1,1) & x^2yz^3\\ 
 3 & \{x^6\}           & (0,0,1) & x^2z^2\\ 
 4 & \{x^7,z^6\}       & (2,0,1) & x^2z\\
 5 & \{x^{12},y^8,z^8\}& \text{--} & 1
 \\ \hline
\end{array}
$$ \qedhere
\end{example}

Using the notation from Proposition \ref{DataProp}, we now define subsets $M_i\subset \{x,y,z\}$ for $1\leq i\leq 2k$ associated to $I$ as follows: $$M_{2i-1}=\{c \mid V_i \text{ contains no pure $c$-power}\} \qquad M_{2i}=\{c\mid c\text{ does not divide } x^{a_i}y^{b_i}z^{c_i}\}$$

We now construct a layered tricolor graph $G_I$ using these subsets. Consider the graph with layers ranging from $1$ to $2k.$ We define the vertex set of this graph as follows: $$V=\{t_c\mid c\in M_t\}.$$The vertex $t_c$ is placed at the layer $t\in \{1, \dots, 2k\}$ with coloring $c\in \{x,y,z\}.$ We construct edges between vertices in layers $t$ and $t+1$ for all $t\in \{1, \dots, 2k-1\}$ if and only if the vertices are of different color. Finally, define the subgraphs $G_I^t$ to be the subgraph obtained by restricting to the first $t$ layers, and define $G_I^0=\varnothing$. 

We say that a subgraph $L\subset G_I$ is \emph{special} if $L$ has at most one vertex at each level, and if $t_c$ and $(t+1)_d$ are vertices in $L,$ they must be connected by an edge. For a special subgraph $L$, we define $e(L)$ to be the number of edges in $L$ and $v(L)$ to be the number of vertices in $L$.

We will now motivate the definitions of these combinatorial statistics. We note that $I$ can be linked to a complete intersection through the method from \ref{thm:hu-criterion} using $2k$ steps in the even case and $2k-1$ steps in the odd case, where $k$ is the integer from Proposition \ref{DataProp}. The links use pure powers in $x$, $y$ and $z$ as the regular sequences, and by Proposition \ref{LinkageCombinatoricsTheorem} the way the decorations change only depends on whether or not the degrees of the pure powers in the regular sequence equal the degrees of the pure powers in the ideal being linked. The sets $M_t$ precisely capture this information, and are sufficient to determine the Herzog class of $I$. The graph $G_I$ is a way of representing this data, and it is convenient to give the closed formula in terms of $G_I.$

\begin{remark}\label{varsepsilonRemark}
     Let
$$
 \varepsilon=
 \begin{cases}
 0,&V_1\ne\{1\},\\
 1,&V_1=\{1\}.
 \end{cases}
$$
In the first case the canonical reverse linkage chain has $2k$ links since the ideal $I^{\{k-1\}}$ is of the form $$\bigl(x^{\alpha_2}, y^{\beta_2}, z^{\gamma_2}, x^{a_1}y^{b_1}z^{c_1}(x^{\alpha_1}, y^{\beta_1},z^{\gamma_1} )\bigr),$$ which is linked to the complete intersection via two links. In the second case it
has $2k-1$ links since the ideal $I^{\{k-1\}}$ is of the form $$(x^{\alpha_2}, y^{\beta_2}, z^{\gamma_2}, x^{a_1}y^{b_1}z^{c_1}),$$ which is linked to the complete intersection via a single link. As a consequence, when $\varepsilon=1$ we delete the $1$st layer from the graph $G_I$ and only consider the layers from $2$ to $2k$.
\end{remark}

\begin{example}\label{ex:graph-new}
For the ideal in Example \ref{ex:monomial-layer-data}, the subsets defined above are
$$
\begin{array}{c|cccccccc}
 t&1&2&3&4&5&6&7&8\\ \hline
 M_t&\{z\}&\varnothing&\varnothing&\{x\}&\{y,z\}&\{x,y\}&\{y\}&\{y\}.
\end{array}
$$
Thus $G_I$ has vertex set
$$
 V=\{1_z,4_x,5_y,5_z,6_x,6_y,7_y,8_y\}
$$
and edge set
$$
 \bigl\{\{4_x,5_y\},\{4_x,5_z\},\{5_y,6_x\},
        \{5_z,6_x\},\{5_z,6_y\},\{6_x,7_y\}\bigr\}.
$$
The graph is drawn below, with layers increasing from top to bottom.
The column labels record the three colors.
\begin{center}
\begin{tikzpicture}[x=1.7cm,y=0.8cm,
  vertex/.style={circle,draw,fill=white,inner sep=1pt,font=\small},
  xvertex/.style={vertex,draw=blue!70!black,fill=blue!8},
  yvertex/.style={vertex,draw=red!70!black,fill=red!8},
  zvertex/.style={vertex,draw=green!45!black,fill=green!8}]
\node at (-1,0) {layer};
\node at (0,0) {$x$};
\node at (1,0) {$y$};
\node at (2,0) {$z$};
\foreach \t in {1,...,8} {
  \node at (-1,-1.5*\t) {$\t$};
  \draw[gray!20,densely dotted] (-0.5,-1.5*\t)--(2.5,-1.5*\t);
}
\node[zvertex] (z1) at (2,-1.5) {$1_z$};
\node[xvertex] (x4) at (0,-6) {$4_x$};
\node[yvertex] (y5) at (1,-7.5) {$5_y$};
\node[zvertex] (z5) at (2,-7.5) {$5_z$};
\node[xvertex] (x6) at (0,-9) {$6_x$};
\node[yvertex] (y6) at (1,-9) {$6_y$};
\node[yvertex] (y7) at (1,-10.5) {$7_y$};
\node[yvertex] (y8) at (1,-12) {$8_y$};
\draw[thick] (x4)--(y5) (x4)--(z5) (y5)--(x6)
             (z5)--(x6) (z5)--(y6) (x6)--(y7);
\end{tikzpicture}
\end{center}
The vertices $1_z$ and $8_y$ are isolated. In particular, no edge joins
$7_y$ to $8_y$, since they have the same color.

An example of a special subgraph is 
$$
 \{1_z,4_x,5_z,6_x,8_y\},
$$
which has $v(L)=5$ and $e(L)=2$. The vertices $5_y,6_y$ cannot
both occur in a special subgraph, since they occupy consecutive layers
but have no edge between them. 
\end{example}

\begin{remark}[Inverse Problem] \label{inverseProblemRemark}
    For every codimension $3$ $\fm$-primary licci monomial ideal $I$ we constructed the sets $M_t$ and the graph $G_I$ associated to it. We can ask the inverse question on whether we can construct an ideal whose associated graph is a given graph. The answer in general is no without some additional assumptions. 

    We first observe that by definition, when $t>1$, $M_t\neq \{x,y,z\}$ must hold. In particular, $|M_t|\leq 2$ holds for all $t>1$ (and also for $t=1$ when $\varepsilon=0$ holds). Further, if $|M_t|=2$ for some $t,$ then $\{x,y,z\}\backslash M_{t}\subset M_{t-1}$ must hold. To see this, suppose that $|M_t|=2$ and $t$ is even. Without loss of generality say $\{x,y,z\}\backslash M_{t}=\{x\}$. Then, the corresponding ideal in the linkage chain looks like $$I_t=(x^\alpha,y^\beta, z^\gamma)+x^aK,$$ where $x^aK$ contains no pure power. Thus, $K$ does not contain a pure $x$-power, which implies that $x\in M_{t-1}.$ If $t$ was odd, then $$I_t=(x^\alpha)+x^ay^bz^cK,$$ and $a=0$ must hold, since the Huneke--Ulrich algorithm implies that the gcd of $I_t$ is $1.$ Thus, $x\in M_{t-1}$ in this case as well.  

    Thus, call the family of subsets $M_t\subset \{x,y,z\}$ for $1\le t\le 2k$ for some positive integer $k$ to be \emph{admissible} if \begin{enumerate}
        \item $|M_t|\leq 2$ for all $1\leq t\leq 2k.$
        \item If $|M_t|=2,$ then $\{x,y,z\}\backslash M_t\subset M_{t-1}$ holds for $1\leq t\leq 2k.$
    \end{enumerate}
    In particular, $|M_1|\leq 1$ must hold. Then, for every admissible family of subsets, there exists a codimension $3$ $\fm$-primary licci monomial ideal $I$ whose associated graph $G_I$ is the one that is defined by the admissible family of subsets. We have already proved one direction of this statement. The other direction follows from a simple (not necessarily unique) construction by recursively defining the generators for $I$ using Proposition \ref{DataProp}.
\end{remark}

\section{Main Result}\label{Section:Mainresult}
\subsection{Recursive Formula}\label{Subsec:rec}
In this section, we obtain a recursive formula for the Herzog class of an $\fm$-primary licci monomial ideal $I$. We retain the data $k,V_i,a_i,b_i,c_i,M_t$ and the graph $G_I$ from
Section \ref{Section:CombinatorialData}. 
For convenience, define $M_{2k+1}=\varnothing$.

\begin{lemma}\label{lem:chain-new}
There is a chain of links
$$
 J_{\varepsilon}\sim J_{\varepsilon+1}\sim\cdots\sim J_{2k}=I,
$$
with $J_{\varepsilon}$ a complete intersection, such that the link at
stage $t$ uses the three minimal pure powers of $J_t$. Among these,
the minimal generators of $J_{t-1}$ are exactly those with colors in
$M_t$. Here, $\varepsilon$ denotes $0$ or $1$ depending on the parity of the number of links (Remark \ref{varsepsilonRemark}). 
\end{lemma}
\begin{proof}
 We construct the intermediate ideals $J_t$ as follows. For $\varepsilon\leq t\leq k$, let $$J_{2t}=I^{\{k-t\}}.$$ From Section \ref{mprimarySection}, we know that $J_{2t}$ and $J_{2(t-1)}$ are linked via two direct links, let $J_{2t-1}$ denote the intermediate ideal. When $\varepsilon=1,$ we need to add one final link between $J_2$ and $J_1,$ where $J_1$ is a complete intersection. This defines the chain of links. From Section \ref{mprimarySection}, we know that the link at stage $t$ is obtained by taking the three minimal pure powers of $J_t.$ Suppose that $$ J_{2(t-1)}=I^{\{k-t+1\}}=(x^{\alpha-a},y^{\beta-b},z^{\gamma-c})+K, \qquad J_{2t}=I^{\{k-t\}}=(x^\alpha,y^\beta,z^\gamma)+x^a y^b z^c K,$$ where $K$ is a proper ideal and its monomial generators have gcd $1$. The ideals are linked via the regular sequences $\{x^{\alpha-a},y^{\beta-b},z^{\gamma-c}\}$ and $\{x^\alpha,y^\beta,z^\gamma\}$.  
 
If $K$ contains a pure $x$-power, its least exponent is strictly less
than $\alpha-a$ by minimality of the generators of $J_{2(t-1)}$. Otherwise
$x^{\alpha-a}$ is a minimal generator of $J_{2(t-1)}$. Also,
$x^{\alpha-a}$ is the smallest pure $x$-power in $J_{2t-1}$. Hence, $x^\alpha$ is a minimal generator of $J_{2t-1}$ if and only if $a=0.$ A similar argument holds for $y$ and $z.$

Consequently, the first link selects precisely the variables missing
from the corresponding set $V_i$, giving $M_{2i-1}$. The second
selects $x$ precisely when $a=0$, and similarly for $y,z$,
giving $M_{2i}$. This completes the proof.

\end{proof}

\begin{theorem}[Recursive formula]\label{thm:rec-new}
Initialize
\[
 \kappa_{\varepsilon}=1,\qquad
 \lambda_{\varepsilon,x}
 =\lambda_{\varepsilon,y}
 =\lambda_{\varepsilon,z}=1.
\]
For $\varepsilon<t\le2k$, define
\begin{align}
 p_t&=\kappa_{t-1}+1-\sum_{c\in M_t}\lambda_{t-1,c},
       \label{eq:new-p}\\
 \lambda_{t,c}
   &=p_t+\begin{cases}\lambda_{t-1,c},&c\in M_t,\\0,&c\notin M_t,\end{cases}
       \label{eq:new-label}\\
 \kappa_t&=\kappa_{t-1}+p_t.\label{eq:new-level}
\end{align}
Then $\kappa(I)=\kappa_{2k}$, and the decoration of $I$ is
\begin{align}
 \lambda(I)
 &=\operatorname{sort}\left(
     \{1\}^{\varepsilon}\cup \{\lambda_{2t,c}\mid \varepsilon\leq t\leq k\text{, }c\in\{x,y,z\}\setminus M_{2t+1}\}
     \right),\label{eq:new-collect-lambda}\\
     \mu(I)
 &=\operatorname{sort}\left(
     \{1\}^{1-\varepsilon}\cup \lambda_{2t-1,c}\mid 1\leq t\leq k\text{, } c\in\{x,y,z\}\setminus M_{2t}\}
     \right)\label{eq:new-collect-mu}. 
\end{align}
Here $\{1\}^0$ denotes the empty multiset and $\{1\}^1=\{1\}$. Also, the sets in \eqref{eq:new-collect-lambda}--\eqref{eq:new-collect-mu} are not really sets but multisets, since repetition of entries is allowed. 
\end{theorem}

\begin{proof}
Consider the chain of links $$
 J_{\varepsilon}\sim J_{\varepsilon+1}\sim\cdots\sim J_{2k}=I
$$ coming from Lemma \ref{lem:chain-new}. Let $\kappa_t$ denote the level of $J_t$ and let $\lambda_{t,c}$ denote the part of the decoration of $J_t$ that corresponds to the minimal pure $c$-power in $J_t$ for $c\in \{x,y,z\}.$ The initial complete intersection has decoration $((1,1,1),(1))$, which explains the initial values. Lemma \ref{lem:chain-new} and Proposition \ref{LinkageCombinatoricsTheorem} tell us how to obtain the decoration of $J_t$ from the decoration of $J_{t-1}$. This gives
\eqref{eq:new-p}--\eqref{eq:new-level}.

It remains to collect the parts. The three labels obtained at stage $t$
are $\lambda_{t,x},\lambda_{t,y},\lambda_{t,z}$ in the first
partition. At stage $t+1$, precisely those indexed by $M_{t+1}$
are removed and the rest move to the second partition $\mu_{J_{t+1}}$. Each surviving
part subsequently alternates between the two partitions in the decorations of the $J_i$'s. It is never
selected again since every later selection uses the three current
pure-power generators, whose labels are the three new labels of
the preceding stage. Thus a surviving label obtained at stage $t$ lies
in the final first partition exactly when $2k-t$ is even.
The initial part $1$ in the second partition makes $2k-\varepsilon$
such switches, placing it in $\mu(I)$ for $\varepsilon=0$ and in
$\lambda(I)$ for $\varepsilon=1$. This proves
\eqref{eq:new-collect-lambda}--\eqref{eq:new-collect-mu}.
\end{proof}

\begin{example}\label{ex:rec-new}
For the ideal of Example \ref{ex:monomial-layer-data}, $\varepsilon=0$ and the recurrence gives
\[
\begin{array}{c|c|r|r|rrr}
 t&M_t&p_t&\kappa_t&\lambda_{t,x}&\lambda_{t,y}&\lambda_{t,z}\\ \hline
 0&\text{--}&\text{--}&1&1&1&1\\
 1&\{z\}&1&2&1&1&2\\
 2&\varnothing&3&5&3&3&3\\
 3&\varnothing&6&11&6&6&6\\
 4&\{x\}&6&17&12&6&6\\
 5&\{y,z\}&6&23&6&12&12\\
 6&\{x,y\}&6&29&12&18&6\\
 7&\{y\}&12&41&12&30&12\\
 8&\{y\}&12&53&12&42&12
\end{array}
\]
At even stages the surviving blocks are
$$
 \{1,1\},\quad\{3,3,3\},\quad\{12\},\quad\{12,6\},
 \quad\{12,42,12\},
$$
at odd stages they are
$$
 \{1,1,2\},\quad\{6,6\},\quad\{12\},\quad\{12,12\}.
$$
Sorting and adjoining the initial $1$ to the second partition yields
$$
 \lambda(I)=(42,12,12,12,12,6,3,3,3,1,1),
$$
$$
 \mu(I)=(12,12,12,6,6,2,1,1,1).
$$
In particular, $\kappa(I)=53$ and
$\lambda(I)$ has eleven parts, one for each minimal generator.
\end{example}

\subsection{Closed Formula}\label{Subsec:Closed}
We now give a closed formula for the recurrence relation in Theorem \ref{thm:rec-new} in terms of the combinatorial data defined in Section \ref{Section:CombinatorialData}. For an $\fm$-primary licci monomial ideal $I,$ let $G_I$ be the corresponding layered graph. We recall that when $\varepsilon=1$, we delete the first layer in $G_I$ (Remark \ref{varsepsilonRemark}). In particular, all the subgraphs $G_I^t$ start from the second layer in this case. For a special subgraph $L\subset G_I$, we define its weight
$$
 w(L)=(-1)^{v(L)}2^{-2v(L)+e(L)}.
$$
Let $\delta_t(L)=1$ if $L$ contains a vertex in layer $t$, and
$\delta_t(L)=0$ otherwise. 

\begin{theorem}[Closed formula]\label{thm:closed-new}
For $\varepsilon\le t\le2k$, the level in
Theorem~\ref{thm:rec-new} satisfies
\begin{equation}\label{eq:closed-kappa-new}
 \kappa_t+1
 =2^{t-\varepsilon+1}
 \sum_{L\subseteq G_I^t\ \mathrm{special}}w(L).
\end{equation}
Let $\lambda_{\varepsilon,c}=1$, and for $\varepsilon<t\le2k$ and $c\in\{x,y,z\}$,
\begin{equation}\label{eq:closed-label-new}
 \lambda_{t,c}
 =2^{t-\varepsilon}
 \sum_{\substack{L\subseteq G_I^t\ \mathrm{special}\\ t_c\notin V(L)}}
 w(L)\,2^{\delta_t(L)}.
\end{equation}
The restriction $t_c\notin V(L)$ is vacuous if $c\notin M_t$.
 Then, a closed formula for the Herzog class of $I$ is obtained by substituting \eqref{eq:closed-label-new} into
\eqref{eq:new-collect-lambda} and \eqref{eq:new-collect-mu}.
\end{theorem}

\begin{proof}
We show directly that the two finite sums satisfy the recurrences
\begin{align}
 \kappa_t+1
 &=2(\kappa_{t-1}+1)-\sum_{d\in M_t}\lambda_{t-1,d},
 \label{eq:proof-kappa-new}\\
 \lambda_{t,c}
 &=\kappa_{t-1}+1-
                 \sum_{d\in M_t\setminus\{c\}}\lambda_{t-1,d}.
 \label{eq:proof-label-new}
\end{align}
These are equivalent to \eqref{eq:new-p}--\eqref{eq:new-level}.
At $t=\varepsilon$, only the empty subgraph occurs, so the first
sum gives $\kappa_{\varepsilon}+1=2$.

Proceed by induction on $t$. Divide the subgraphs in
the summation of \eqref{eq:closed-kappa-new} according to whether they meet layer $t$.
The sum over those that do not meet it is $2(\kappa_{t-1}+1)$ by the induction
hypothesis.

Now, if a special subgraph contains $t_d$, where $d$ is a fixed color, delete that vertex to obtain
$L'$. The subgraph $L'$ cannot contain $(t-1)_d$. Conversely, every
eligible special $L'\subseteq G_I^{t-1}$ that does not contain $(t-1)_d$ extends
uniquely in this way. Adding $t_d$ contributes one vertex and
$\delta_{t-1}(L')$ edges, so
\[
 w(L'\cup\{t_d\})
 =-\frac14\,w(L')\,2^{\delta_{t-1}(L')}.
\]
Using the induction hypothesis for \eqref{eq:closed-label-new}, we see that the sum over all $L'$ along with the factor $2^{t-\varepsilon+1}$ gives $-\lambda_{t-1,d}$.
We note that for $t=\varepsilon+1$, the same calculation uses only the empty
$L'$ and $\lambda_{\varepsilon,d}=1$. Finally, 
summing over $d\in M_t$ proves \eqref{eq:proof-kappa-new}.

For \eqref{eq:closed-label-new}, the sum over subgraphs avoiding layer $t$
contribute $\kappa_{t-1}+1$. For the remaining sum, a subgraph meeting layer $t$ can contain $t_d$ precisely when $d\in M_t\setminus\{c\}$. The additional vertex along with $2^{\delta_t(L)}$ contributes a factor of
$-\frac{1}{2}$. Together with the coefficient $2^{t-\varepsilon}$, summing over subgraphs that contain $t_d$ 
contribute a total of $-\lambda_{t-1,d}$. Finally, summing over $d\in M_t\backslash \{c\}$ proves \eqref{eq:proof-label-new}, completing the induction argument.
\end{proof}

\begin{example}
We return to the ideal from Example \ref{ex:monomial-layer-data} and compute its
decoration using the closed formula. Recall that
\[
(M_1,\ldots,M_8)
=
(\{z\},\varnothing,\varnothing,\{x\},
 \{y,z\},\{x,y\},\{y\},\{y\}).
\]
Here $\varepsilon=0$. We illustrate the calculation by computing
$\lambda_{8,y}$.

Since $8_y$ is the only vertex in layer $8$, the restriction
$8_y\notin V(L)$ forces $L$ to lie in $G_I^7$. In particular,
$\delta_8(L)=0$, so the closed formula becomes
\[
\lambda_{8,y}
=
2^8
\sum_{\substack{L\subseteq G_I^7\\L\text{ special}}}
(-1)^{v(L)}2^{-2v(L)+e(L)}.
\]
The vertex $1_z$ is isolated and is separated from the remaining
vertices by two empty layers. Thus including or omitting $1_z$
contributes a factor of $1-\frac14=\frac34$.

Let $H$ be the induced subgraph on layers $4,\ldots,7$.
For each $j$, grouping its special subgraphs by their number
of vertices gives
\[
\begin{array}{c|ccccc}
j&0&1&2&3&4\\ \hline
\displaystyle
\sum_{\substack{L\subseteq H\text{ special}\\v(L)=j}}
2^{e(L)}
&1&6&17&26&16.
\end{array}
\]
For example, among the special subgraphs with two vertices,
five have no edge and six have one edge, giving
$5+6\cdot2=17$. Consequently,
\[
\begin{aligned}
\lambda_{8,y}
&=
2^8\cdot\frac34
\left(
1-\frac64+\frac{17}{16}
-\frac{26}{64}+\frac{16}{256}
\right)\\
&=
2^8\cdot\frac34\cdot\frac7{32}
=42.
\end{aligned}
\]

The remaining parts follow similarly from the closed formula. After computing all the parts, we obtain

\[
\begin{aligned}
\lambda(I)&=(42,12,12,12,12,6,3,3,3,1,1),\\
\mu(I)&=(12,12,12,6,6,2,1,1,1).
\end{aligned}
\]
This coincides with the calculation in Example \ref{ex:rec-new}.
\end{example}

In \cite{LinkageSchutFuctors}, the authors show that for fixed format $(1,6,8,3),$ there are infinitely many Herzog classes of licci ideals having minimal free resolution of this format. We show that this is not the case when we restrict to $\fm$-primary monomial ideals.

\begin{corollary}
    For a fixed number of minimal generators, only finitely many Herzog classes occur among codimension $3$ $\fm$-primary licci monomial ideal. 
\end{corollary}
\begin{proof}
    Let $b$ denote the minimal number of generators. Let $G_I$ denote the graph of an $\fm$-primary licci monomial ideal generated minimally by $b$ elements. By Theorem \ref{thm:closed-new}, it suffices to show that there are only finitely many possibilities for $G_I$. 

    Now, the number of layers in $G_I$ is bounded above by $2b,$ since each double link in the Huneke--Ulrich algorithm gets rid of at least one generator. Since each layer has at most two vertices, the total number of vertices in $G_I$ is bounded above by $4b.$ In particular, this implies that there are only finitely many graphs possible for given $b$, which completes the proof.
\end{proof}

\begin{example}[Generic linkage]

    In the case when the graph $G_I$ is empty, i.e., when $M_i$ is empty for all $i,$ Theorem \ref{thm:closed-new} yields  

    \[
\begin{aligned}
\lambda(I)&=(2^{2k},2^{2k},2^{2k},2^{2k-2},2^{2k-2},2^{2k-2},\dots 1,1,1),\\
\mu(I)&=(2^{2k-1},2^{2k-1},2^{2k-1},2^{2k-3},2^{2k-3},2^{2k-3},\dots ,2,2,2,1)
\end{aligned}
\] when $V_1\neq (1),$ and $$\begin{aligned}
\lambda(I)&=(2^{2k-1},2^{2k-1},2^{2k-1},2^{2k-3},2^{2k-3},2^{2k-3},\dots 2,2,2,1),\\
\mu(I)&=(2^{2k-2},2^{2k-2},2^{2k-2},2^{2k-4},2^{2k-4},2^{2k-4},\dots 1,1,1)
\end{aligned}
$$ when $V_1=(1)$. This happens precisely when each of the links that connect $I$ to a complete intersection is a generic link.  
\end{example}

\begin{example}
Another case that simplifies significantly is when $D$ has no consecutive elements.  This happens when at least one link out of any two given consecutive links is generic.

Every special subgraph is obtained by independently choosing either
no vertex or one vertex from each occupied layer. In this situation, by Remark \ref{inverseProblemRemark} every occupied layer has exactly one vertex. Since there are no edges, Theorem \ref{thm:closed-new} gives
\[
 \kappa_t+1
 =
 2^{t-\varepsilon+1}
 \prod_{j=\varepsilon+1}^{t}
 \left(1-\frac{|M_j|}{4}\right)
\] since you either pick nothing from the $j$th layer, which corresponds to choosing the $1$ for that respective term, or pick a vertex in $M_j,$ which contributes a factor of $\frac{-1}{4}.$

Similarly, for $\varepsilon<t\le 2k$,
\[
 \lambda_{t,c}
 =
 2^{t-\varepsilon}
 \prod_{j=\varepsilon+1}^{t-1}
 \left(1-\frac{|M_j|}{4}\right)
 \left(1-\frac{|M_t\setminus\{c\}|}{2}\right).
\]

Let $d_t$ denote the number of non-empty layers between $\varepsilon+1$ and $t.$ The level therefore simplifies to
\[
 \kappa(I)+1=2^{2k-\varepsilon+1-2d_{2k}}3^{d_{2k}}.
\]
For $\varepsilon<t\le 2k$, set
\[
 Q_t=2^{t-\varepsilon-2d_{t-1}}3^{d_{t-1}}.
\]
If $M_t=\varnothing$, all three labels at stage $t$ equal $Q_t$.
If $M_t=\{c\}$, the label of color $c$ equals $Q_t$, while
the other two labels equal $Q_t/2$.

We can now write the decoration explicitly. Let $B_t$ denote
the multiset of labels generated at stage $t$ that survive the next
link. Then,
\[
 B_\varepsilon=\{1\}^{\,3-|M_{\varepsilon+1}|},
 \qquad
 B_t=
 \begin{cases}
   \{Q_t\}^{\,3-|M_{t+1}|},
       & M_t=\varnothing,\\[2pt]
   \{Q_t,Q_t/2,Q_t/2\},
       & M_t\ne\varnothing.
 \end{cases}
\]
Here $\{1\}^r$ denotes the multiset containing $r$ copies of $1$.

Consequently,
\[
 \lambda(I)=
 (B_{2k}, B_{2k-2},\dots, B_{2\varepsilon},1^\varepsilon)
 \qquad
 \mu(I)= (B_{2k-1}, B_{2k-3},\dots, B_1,1^{1-\varepsilon}) 
\]
In particular, each part is a product of powers of $2$ and $3.$
\end{example}

\bibliographystyle{alpha}
\bibliography{reference}
\end{document}